\documentclass[12pt]{amsart}
\usepackage{etex}
\usepackage{amsmath,amsthm,amssymb,stmaryrd,bm,graphicx,float,amssymb,verbatim,pst-all,color,xypic,tikz-cd}
\usepackage[alphabetic]{amsrefs}
\usepackage[all]{xy}
\usepackage[margin=1.0in]{geometry}
\usepackage{mathptmx}
\usepackage[foot]{amsaddr}
\ifx\JPicScale\undefined\def\JPicScale{1.0}\fi
\unitlength \JPicScale mm
{\theoremstyle{plain}
   
   \newtheorem{theorem}{Theorem}[section]
   \newtheorem{proposition}[theorem]{Proposition}     
   \newtheorem{lemma}[theorem]{Lemma}
   
   \newtheorem{corollary}[theorem]{Corollary}
   
   \newtheorem{question}[theorem]{Question}
}
{\theoremstyle{definition}
   \newtheorem*{ack}{Acknowledgements}
   
   \newtheorem{example}[theorem]{Example}
   
   \newtheorem{remark}[theorem]{Remark}

}

\newcommand{\RR}{\mathbb{R}}

\newcommand{\PP}{\mathbb{P}}

\newcommand{\quotient}{/\hspace{-1.2mm}/}
\newcommand{\ChowQ}{/\hspace{-1.2mm}/_{Ch}}

\newcommand{\M}{\overline{M}}

\newcommand{\PGL}{\operatorname{PGL}}
\newcommand{\Chow}{\operatorname{Chow}}

\newcommand{\Gr}{\operatorname{Gr}}

\newcommand{\Proj}{\operatorname{Proj}}
\newcommand{\Spec}{\operatorname{Spec}}

\newcommand{\Hom}{\operatorname{Hom}}

\newcommand{\Conv}{\operatorname{Conv}}
\newcommand{\Bl}{\operatorname{Bl}}
\newcommand{\GL}{\operatorname{GL}}

\newcommand{\Ga}{\mathbb{G}_a}
\newcommand{\Gm}{\mathbb{G}_m}

\newcommand{\Aut}{\operatorname{Aut}}
\newcommand{\vol}{\operatorname{vol}}

\numberwithin{theorem}{section}

\begin{document}

\title{Chow quotients of Grassmannians by diagonal subtori}
\author{Noah Giansiracusa$^*$} 
\author{Xian Wu$^{**}$}
\address{$^*$Assistant Professor of Mathematical Sciences, Bentley University} 
\address{$^{**}$Postdoc, Jagiellonian University, Poland}
\email{ngiansiracusa@bentley.edu, xianwu@uga.edu}

\begin{abstract}
The literature on maximal torus orbits in the Grassmannian is vast; in this paper we initiate a program to extend this to diagonal subtori.  Our main focus is generalizing portions of Kapranov's seminal work on Chow quotient compactifications of these orbit spaces.  This leads naturally to discrete polymatroids, generalizing the matroidal framework underlying Kapranov's results.  By generalizing the Gelfand-MacPherson isomorphism, these Chow quotients are seen to compactify spaces of arrangements of parameterized linear subspaces, and a generalized Gale duality holds here.  A special case is birational to the Chen-Gibney-Krashen moduli space of pointed trees of projective spaces, and we show that the question of whether this birational map is an isomorphism is a specific instance of a much more general question that hasn't previously appeared in the literature, namely, whether the geometric Borel transfer principle in non-reductive GIT extends to an isomorphism of Chow quotients.
\end{abstract}

\maketitle

\section{introduction}

The literature on maximal torus orbits in the Grassmannian and the torus-equivariant geometry (cohomology, K-theory, etc.) of the Grassmannian is extensive; it is a rich field beautifully interweaving combinatorics, representation theory, and geometry, with many applications across these disciplines.  One of the seminal works is Kapranov's paper on Chow quotients in which he compactifies the space of maximal torus orbit closures \cite{Kap93Chow}.  The goal of the present paper is to initiate a program of studying diagonal subtorus orbits in the Grassmannian; we focus here on extending portions of Kapranov's paper to this setting and explore some consequences.

\subsection{Setup and notation}
Fix a base field $k$.  By a \emph{diagonal subtorus} $S$ we mean that coordinates in the maximal torus $T=(k^\times)^n$ acting on $\Gr(d,n)$ are allowed to coincide; that is, $S = (k^\times)^m$ for $m \le n$ and we have an inclusion map $S \hookrightarrow T$ given by a matrix whose rows are all standard basis vectors.  Up to permutation, every such subtorus is of the form \[S = \{(\underbrace{t_1,\ldots,t_1}_{r_1},\underbrace{t_2,\ldots,t_2}_{r_2},\ldots,\underbrace{t_m,\ldots,t_m}_{r_m}) ~|~ t_i \in k^\times\} \subseteq T,\] where $\sum r_i = n$.  Setting $r_i = 1$ for all $i$ recovers Kapranov's case of the maximal torus.  In essence, the combinatorics in Kapranov's paper (matroids, matroid subdivisions, etc.) are generalized by replacing the set $[n] = \{1,2,\ldots,n\}$ with the multiset \[[\vec{r}] := \{\underbrace{1,\ldots,1}_{r_1},\underbrace{2,\ldots,2}_{r_2},\ldots,\underbrace{m,\ldots,m}_{r_m}\}\] where $i$ has multiplicity $r_i$.  (Matroids on multisets appear in the literature under the name discrete polymatroids \cite{HH02}.)  The hypersimplex $\Delta(d,n) \subseteq \RR^n$, a polytope playing a fundamental role in Kapranov's paper, is replaced with its projection under the linear map \[\lambda_{\vec{r}} : \RR^n \rightarrow \RR^m\] given by the matrix $|e_1~\cdots~e_1~e_2~\cdots~e_2~\cdots~e_m~\cdots~e_m|$, the transpose of the matrix defining the inclusion $S \hookrightarrow T$.  These vague assertions will be made precise in what follows.

\subsection{Results}

The $T$-orbit closure of any $k$-point of $\Gr(d,n) \subseteq \PP^{\binom{n}{d}-1}$ is a polarized toric variety whose corresponding lattice polytope is a subpolytope of the hypersimplex
\[\Delta(d,n) = \{(a_1,\ldots,a_n)~|~a_i\in[0,1]\text{ and }\sum a_i = d \} \subseteq \RR^n.\]  This subpolytope has its vertices and edges among those of $\Delta(d,n)$; subpolytopes with this property are called matroid polytopes and are known to be in bijection with rank $d$ matroids on $[n]$, with matroids representable over $k$ identified with the polytopes of $T$-orbit closures in $\Gr(d,n)$  \cite{GGMS87}.  This perspective of matroid polytopes is a relatively recent advance in matroid theory that has fruitfully brought the subject closer to algebraic geometry (cf. \cite[\S1]{Volume}).  Via diagonal subtori, this story extends seamlessly to discrete polymatroids:

\begin{theorem}\label{thm:matroid}
Rank $d$ discrete polymatroids on the multiset $[\vec{r}]$ are in bijection with subpolytopes of $\lambda_{\vec{r}}(\Delta(d,n)) \subseteq \RR^m$ whose vertices and edges are among the images of the vertices and edges of $\Delta(d,n)$; moreover, this bijection identifies the discrete polymatroids representable over $k$ with the lattice polytopes corresponding to $S$-orbit closures in $\Gr(d,n)$.
\end{theorem}

Now let $k = \mathbb{C}$.  Kapranov's idea for compactifying the space of maximal torus orbit closures in $\Gr(d,n)$ is to take a sufficiently small $T$-invariant Zariski open locus $U\subseteq \Gr(d,n)$ such that the $T$-action on $U$ is free and there is an inclusion $U/T \hookrightarrow \Chow(\Gr(d,n))$ sending each torus orbit to its Zariski closure, viewed as an algebraic cycle on the Grassmannian.  The closure of the image of this embedding in the Chow variety is by definition the Chow quotient $\Gr(d,n)\ChowQ T$ \cite{Kap93Chow}.  We can apply the same idea here and study the diagonal subtorus Chow quotient $\Gr(d,n)\ChowQ S$.  We compute some explicit examples of this Chow quotient, together with its natural closed embedding in the toric Chow quotient $\PP^{\binom{n}{d}-1}\ChowQ S$, in \S\ref{sec:examples}.

Kapranov shows \cite[Theorem 1.6.6]{Kap93Chow} that the rational maps sending a linear space to its intersection with, and projection onto, a coordinate hyperplane induce morphisms \[\Gr(d,n)\ChowQ T \rightarrow \Gr(d-1,n-1)\ChowQ T'\text{ and }\Gr(d,n)\ChowQ T \rightarrow \Gr(d,n-1)\ChowQ T',\] respectively, where $T' = (k^\times)^{n-1}$ is the maximal torus acting on these smaller Grassmannians.  We have the following extension of this to the subtorus setting:

\begin{theorem}\label{thm:hypersimplicial}
Fix an index $1 \le i \le m$, let $I \subseteq [n]$ index the $r_i$ coordinates of $S$ corresponding to the $i^{\text{th}}$ $\Gm$-factor, and let $S_i$ denote the rank $m-1$ torus given by projecting $S$ onto the complement of the $I$-coordinates.  Then intersection with, and projection onto, the codimension $r_i$ coordinate linear space defined by $x_j = 0$ for all $j\in I$ induce morphisms
\[\Gr(d,n)\ChowQ S \rightarrow \Gr(d-r_i,n-r_i)\ChowQ S_i \text{ and } \Gr(d,n)\ChowQ S \rightarrow \Gr(d,n-r_i)\ChowQ S_i,\] 
respectively.
\end{theorem}

Kapranov's proof directly analyzes Chow forms to demonstrate their polynomial dependence, whereas we use polytopal subdivisions to apply a valuative criterion for regularity; thus, we obtain in particular a new variant of Kapranov's proof in the case of the maximal torus.

The Gelfand-MacPherson correspondence identifies generic torus orbits in the Grassmannian with generic general linear group orbits in a product of projective spaces, and Kapranov shows \cite[Theorem 2.2.4]{Kap93Chow} that this extends to an isomorphism of Chow quotients \[\Gr(d,n)\ChowQ T \cong (\PP^{d-1})^n\ChowQ\GL_d.\]  Thus, his Grassmannian Chow quotient can be viewed as compactifying the space of configurations of points in projective space, up to projectivity, or dually, the space of hyperplane arrangements.  This has been a fruitful perspective \cite{HKT,Alexeev-weighted} and it generalizes to our setting as follows:

\begin{theorem}\label{thm:genGM}
There is an isomorphism
\[\Gr(d,n)\ChowQ S \cong \left(\prod_{i=1}^m \PP\Hom(k^{r_i},k^d)\right)\ChowQ\GL_d\]
where $\GL_d$ acts diagonally by left matrix multiplication.
\end{theorem}

To prove this, we adapt an argument of Thaddeus in \cite{Tha99} and so also obtain a new proof of Kapranov's original result as a special case.  We can view the right side of the above isomorphism as compactifying the space of arrangements of ``parameterized''  linear subspaces: $(L_1,\alpha_1,\ldots,L_m,\alpha_m)$ where $L_i \subseteq \PP^{d-1}$ is a linear subspace of dimension $r_i - 1$ and $\alpha_i \in \Aut(L_i) \cong \PGL_{r_i}$.

Since orthogonal complement yields a $T$-equivariant isomorphism $\Gr(d,n) \cong \Gr(n-d,n)$ and hence an isomorphism of Chow quotients $\Gr(d,n)\ChowQ S \cong \Gr(n-d,n)\ChowQ S$ for any diagonal subtorus $S\subseteq T$, our generalized Gelfand-MacPherson isomorphism implies the following generalized Gale duality:

\begin{corollary}\label{cor:genGale}
There is a natural involutive isomorphism
\[\left(\prod_{i=1}^m \PP\Hom(k^{r_i},k^d)\right)\ChowQ\GL_d \cong \left(\prod_{i=1}^m \PP\Hom(k^{r_i},k^{(\sum_{i=1}^m r_i)-d})\right)\ChowQ\GL_{(\sum_{i=1}^m r_i)-d}\]
\end{corollary}

In geometric terms, arrangements (up to projectivity) of $m$ generic parameterized linear subspaces $L_i \hookrightarrow \PP^{d-1}$ and their Chow limits are in natural bijection with arrangements (up to projectivity) of $m$ generic parameterized linear subspaces, of the same dimensions, in $\PP^{m -d - 1 + \sum\dim(L_i)}$ and their Chow limits.

Kapranov showed \cite[Theorem 4.1.8]{Kap93Chow} that his Chow quotients generalize the ubiquitous Grothendieck-Knudsen moduli spaces of stable pointed rational curves, namely 
\[Gr(2,n)\ChowQ \GL_2 \cong \M_{0,n}.\] Another generalization was constructed by Chen-Gibney-Krashen in \cite{CGK05}, where a moduli space denoted $T_{d,n}$ compactifying the space of $n$ distinct points and a disjoint parameterized hyperplane in $\PP^d$ up to projectivity was introduced and studied and shown to satisfy $T_{1,n} \cong \M_{0,n+1}$.  Essentially $T_{d,n}$ is the locus in the Fulton-MacPherson configuration space $X[n]$ \cite{Fulton-MacPherson} where all $n$ points have come together at a single fixed smooth point on a $d$-dimensional variety $X$ \cite[\S3.1]{CGK05}.  The space $T_{d,n}$ is birational to $\Gr(d+1,n+d)\ChowQ S$, where \[S = \{(\underbrace{t_1,\ldots,t_1}_{d},t_2,\ldots,t_{n+1})~|~t_i\in k^\times\},\] since both compactify the space of $n$ distinct points and a disjoint parameterized hyperplane $\PP^{d-1} \hookrightarrow \PP^d$ up to projectivity.  Krashen has asked, informally, whether this birational map is actually an isomorphism.  While we have not been able to answer this question, we conclude this paper by showing that Krashen's question is a specific instance of a much more general question that appears not to have been asked previously in the literature---namely, whether the classical Borel transfer principle (relating non-reductive invariants to reductive invariants) extends from GIT quotients \cite{Doran-Kirwan} to Chow quotients.  

\begin{ack}
We thank Gary Gordon and Felipe Rincon for drawing our attention to discrete polymatroids, and we thank Valery Alexeev, Danny Krashen, and Angela Gibney for helpful conversations on this project.  This paper is part of the second author's PhD dissertation at the University of Georgia, supervised by the first author.  The first author was supported in part by NSF grant DMS-1802263, NSA grant H98230-16-1-0015, and Simons Collaboration Grant 346304.
\end{ack}

\section{Discrete polymatroids}

For a non-negative integer vector $v = (v_1,\ldots,v_m) \in \mathbb{Z}_{\ge 0}^m$, the \emph{modulus} is $|v| = \sum v_i$.  A \emph{discrete polymatroid} on the ground set $[m] = \{1,2,\ldots,m\}$ can be defined as a nonempty finite subset $B \subseteq \mathbb{Z}_{\ge 0}^m$ of vectors all of the same modulus (called the \emph{rank} of $B$) satisfying the following exchange property: if $u,v\in B$ with $u_i > v_i$ for some $1 \le i \le m$, then there exists $1 \le j \le m$ such that $u_j < v_j$ and $u - e_i + e_j \in B$ \cite[Theorem 2.3]{HH02}.  This can be reformulated in terms of multisets as follows.  Given a discrete polymatroid $B$, let \[[\vec{r}] := \{\underbrace{1,\ldots,1}_{r_1},\underbrace{2,\ldots,2}_{r_2},\ldots,\underbrace{m,\ldots,m}_{r_m}\}\] be the multiset where $i$ has multiplicity $r_i := \max_{v\in B}\{v_i\}$.  Each element of $B$ can then be viewed as a sub-multiset of $[\vec{r}]$.  If one considers the usual basis definition of a matroid except replacing the word ``set'' with ``multiset'' then the discrete polymatroid $B$ is a matroid on the multiset $[\vec{r}]$, and conversely any matroid on a multiset is a discrete polymatroid on the ground set given by the set underlying the multiset.  We will freely switch between the multiset perspective and the integer vector perspective of discrete polymatroids.

\begin{proof}[Proof of Theorem \ref{thm:matroid}]
This can either be proven by adapting the original arguments in \cite{GGMS87}, or it can be reduced to the results in \cite{GGMS87} by using a multiset projection map; we present here the latter approach.  

Fix an integer $d\ge 1$ and a multiset $[\vec{r}]$ with underlying set $[m] = \{1,2,\ldots,m\}$ where $i$ has multiplicity $r_i \ge 1$.  Let $\pi_{\vec{r}} : [n] \rightarrow [m]$ be the ``projection'' map sending $1,2,\ldots,r_1$ to 1, and $r_1+1,\ldots,r_1+r_2$ to 2, etc.  By a slight abuse of notation, for a subset $A = \{a_1,\ldots,a_\ell\}\subseteq [n]$ we denote by $\pi_{\vec{r}}(A)$ the multiset $\{\pi_{\vec{r}}(a_1),\ldots,\pi_{\vec{r}}(a_\ell)\}$, in other words the multiplicity of $j$ is the cardinality of the fiber $\pi_{\vec{r}}^{-1}(j) \cap A$.  Clearly $\pi_{\vec{r}}$ then sends a rank $d$ matroid on $[n]$ to a rank $d$ discrete polymatroid on $[m]$, and conversely if $B$ is a rank $d$ discrete polymatroid on $[m]$ then $\{A \subseteq [n]~|~\pi_{\vec{r}}(A) \in B\}$ is a rank $d$ matroid on $[n]$; we denote the latter matroid by $\pi_{\vec{r}}^{-1}(B)$.

Given a rank $d$ discrete polymatroid $B$ on the multiset $[\vec{r}]$, the rank $d$ matroid $\pi_{\vec{r}}^{-1}(B)$ on $[n]$ has basis polytope $P$ given by the convex hull of the vectors $e_A := \sum_{i\in A} e_i$ for $A \in \pi_{\vec{r}}^{-1}(B)$, and by the classical results of \cite{GGMS87} the vertices and edges of $P$ are among the vertices and edges of the hypersimplex $\Delta(d,n)$.  It then follows trivially that the linear projection $\lambda_{\vec{r}}(P)$ has its vertices and edges among the images under $\lambda_{\vec{r}}$ of the vertices and edges of $\Delta(d,n)$.  Moreover, $\lambda_{\vec{r}}(P) \subseteq \mathbb{R}^m$ is the convex hull of the basis vectors of $B$ (where now we view $B$ as a set of vectors in $\mathbb{Z}^m_{\ge 0}$), and by \cite[Theorem 3.4]{HH02} we can recover $B$ from this convex hull (specifically, the integral vectors in this convex hull are the independent sets in $B$).  This faithfully embeds the set of rank $d$ discrete polymatroids on $[\vec{r}]$ into the set of subpolytopes of $\lambda_{\vec{r}}(\Delta(d,n))$ whose vertices and edges are among the images under $\lambda_{\vec{r}}$ of those of $\Delta(d,n)$. This association is also surjective, since if $Q \subseteq \lambda_{\vec{r}}(\Delta(d,n))$ is a subpolytope whose vertices and edges are among the images of those of $\Delta(d,n)$, then the preimage of $Q$ under $\lambda_{\vec{r}}$ is a subpolytope of $\Delta(d,n)$ whose vertices and edges are among those of this hypersimplex, i.e., $\lambda_{\vec{r}}^{-1}(Q)$ is a matroid polytope, and the multiset image under $\pi_{\vec{r}}$ of the corresponding rank $d$ matroid on $[n]$ is a rank $d$ discrete polymatroid with $Q$ as its associated polytope.

We now turn to the assertion about representability.  Given a $k$-point of the Grassmannian $L\in \Gr(d,n)(k)$, the lattice polytope $\Delta_{\overline{S\cdot L}}$ for the projective toric variety $\overline{S\cdot L} \subseteq \PP^{\binom{n}{d}-1}$ is the image of this torus orbit closure under the moment map $\mu_S : \PP^{\binom{n}{d}-1} \rightarrow \mathbb{R}^m$ for $S$.  This moment map is the composition of the moment map $\mu_T : \PP^{\binom{n}{d}-1} \rightarrow \mathbb{R}^n$ for the maximal torus $T$ with the linear projection $\lambda_{\vec{r}} : \mathbb{R}^n \rightarrow \mathbb{R}^m$.  Thus, 
\[\Delta_{\overline{S\cdot L}} = \mu_S(\overline{S\cdot L}) = \lambda_{\vec{r}}\left(\mu_T(\overline{S\cdot L})\right) = \lambda_{\vec{r}}\left(\mu_T(\overline{T\cdot L})\right) = \lambda_{\vec{r}}(\Delta_{\overline{T\cdot L}}),\]
which is the polytope associated to the discrete polymatroid $\pi_{\vec{r}}(M(L))$, where $M(L)$ is the matroid represented by  $L$.  But $\pi_{\vec{r}}(M(L))$ is also the discrete polymatroid represented by $L$.
\end{proof}

The linear projection $\lambda_{\vec{r}} : \mathbb{R}^n \rightarrow \mathbb{R}^m$ may send vertices of the hypersimplex $\Delta(d,n)$ to non-vertex points of the polytope $\lambda_{\vec{r}}(\Delta(d,n))$, and for the above theorem it is crucial that our subpolytopes are allowed to use such points rather than just the actual vertices of $\lambda_{\vec{r}}(\Delta(d,n))$, as the following example illustrates:

\begin{example}
Let $\vec{r} = (1,2,2)$, so $n=5$ and $m=3$; the projection function $\pi_{\vec{r}}$ is $1 \mapsto 1$, and $2,3 \mapsto 2$, and $4,5 \mapsto 3$; in coordinates, the linear projection $\lambda_{\vec{r}} : \mathbb{R}^5 \rightarrow \mathbb{R}^3$ is $(x_1, x_2+x_3, x_4+x_5)$.  Consider rank 3 matroids.  The hypersimplex $\Delta(3,5)$ has 10 vertices, the permutations of the vector $(1,1,1,0,0)$; the images of these 10 vertices are $(1,1,1)$ four times, $(1,2,0)$ once, $(1,0,2)$ once, $(0,1,2)$ twice, and $(0,2,1)$ twice.  The polytope $\lambda_{\vec{r}}(\Delta(3,5))$ is a trapezoid, and the point $(1,1,1)$ is not a vertex of this trapezoid even though it is the image of vertices of the hypersimplex (see Figure \ref{fig:12233}).  The segment from, say, $(1,1,1)$ to $(1,2,0)$ is a discrete polymatroid even though it has a vertex that is not a vertex of the trapezoid.  On the other hand, the four vertices of the trapezoid $(1,2,0), (1,0,2), (0,1,2), (0,2,1)$ do not form a discrete polymatroid because the trapezoid edge from $(1,2,0)$ to $(1,0,2)$ is not an edge of the projected hypersimplex, it is a union of two such edges (and indeed the basis exchange axiom fails on these two without the presence of the midpoint $(1,1,1)$).
\end{example}

\begin{figure}
\centering
\begin{tikzpicture}
	% axis
	\draw[->] (xyz cs:x=0) -- (xyz cs:x=4) node[below] {$y$};
	\draw[->] (xyz cs:y=0) -- (xyz cs:y=4) node[right] {$z$};
	\draw[->] (xyz cs:z=0) -- (xyz cs:z=6) node[left] {$x$};
	%dashed lines
	\draw[dashed,blue] (3,0) -- (0,3);
	\draw[dashed,blue] (0,3) -- (-1.5,-1.5);
	\draw[dashed,blue] (3,0) -- (-1.5,-1.5);
	%solid lines
	\draw[] (-0.5,1.5) node[left] {(1,0,2)} -- (1,2) node[above right] {(0,1,2)} -- (2,1) node[above right] {(0,2,1)}-- (1.5,-0.5) node[below right]{(1,2,0)} -- (-0.5,1.5) -- (0.5,0.5) node[left] {(1,1,1)};
	%node
	\foreach \Point in
	{(-0.5,1.5) , (1,2) , (2,1) , (1.5,-0.5) , (-0.5,1.5), (0.5,0.5)}
	{\node at \Point {\textbullet};}
\end{tikzpicture}
\caption{For the multiset $\{1,2,2,3,3\}$ the projected polytope $\lambda_{\vec{r}}(\Delta(3,5))$ is a trapezoid lying on a triangle.  The point $(1,1,1)$ is not a vertex of the trapezoid even though it is the image under the linear projection $\lambda_{\vec{r}} : \mathbb{R}^5 \rightarrow \mathbb{R}^3$ of a vertex (in fact, four of them) of the hypersimplex $\Delta(3,5)$.}
\label{fig:12233}
\end{figure}
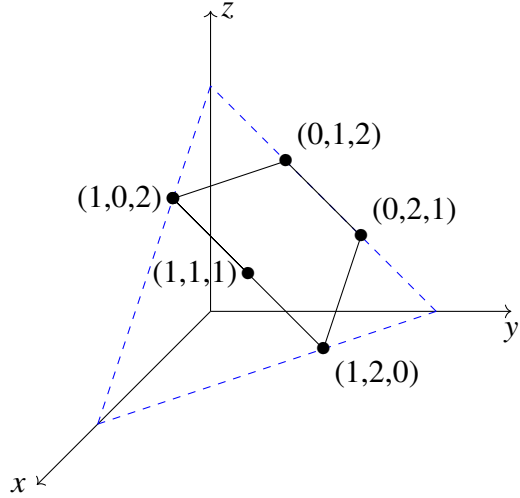

The interior of the Chow quotient $\Gr(d,n)\ChowQ S$ consists, by definition, of torus orbit closures $\overline{S\cdot L}$ (viewed as algebraic cycles) for generic linear subspaces $L\in Gr(d,n)$; taking the closure of this interior locus in the Chow quotient adds limit points that are certain algebraic cycles \[\sum_{i=1}^\ell m_iZ_i \in \Chow\left(\Gr(d,n)\right),\] about which, following Kapranov, we can now say a bit more (cf. \cite[Proposition 1.2.11]{Kap93Chow}):

\begin{proposition}\label{prop:cycles}
For each cycle $\sum_{i=1}^\ell m_i Z_i \in \Gr(d,n)\ChowQ S$, the multiplicities $m_i$ are all 1 and the irreducible cycles $Z_i$ are single orbit closures $\overline{S\cdot L_i}$, $L_i\in \Gr(d,n)$.  The lattice polytopes $\Delta_{\overline{S\cdot L_i}}$, for $i=1,\ldots,\ell$, form a polyhedral decomposition of $\lambda_{\vec{r}}(\Delta(d,n))$.
\end{proposition}

\begin{proof}
The condition in Kapranov's \cite[Theorem 0.3.1]{Kap93Chow} is automatically satisfied here so each $Z_i$ is a single orbit closure $\overline{S\cdot L_i}$.  Kapranov's proof of \cite[Proposition 1.2.15]{Kap93Chow} shows that the index of the sub-lattice generated by the vertices of the representable matroid polytope $\Delta_{\overline{T\cdot L_i}}$ inside the lattice generated by the vertices of the hypersimplex $\Delta(d,n)$ is one.  This index is preserved when applying the linear map $\lambda_{\vec{r}}$, and as we noted at the end of the proof of Theorem \ref{thm:matroid} we have $\Delta_{\overline{S\cdot L_i}} = \lambda_{\vec{r}}(\Delta_{\overline{T\cdot L_i}})$, so we have that the multiplicity $m_i$ of our cycle $\overline{S\cdot L_i}$ is also one.  The assertion about polyhedral decompositions follows the more general result \cite[Proposition 3.6]{KSZ91}, since the torus equivariant Pl\"ucker embedding identifies each point of the Chow quotient $\Gr(d,n)\ChowQ S$ with a point of the toric Chow quotient $\PP^{\binom{n}{d}-1}\ChowQ S$.
\end{proof}

\section{Examples of subtorus Chow quotients}\label{sec:examples}

In this section we describe some diagonal subtorus Chow quotients of $\Gr(2,4)$, starting with the case of the maximal torus that Kapranov worked out in \cite[Example 1.2.12]{Kap93Chow} so that we can present an explicit equational approach that generalizes to the other cases.  First, let us recall the more general setup.  The Pl\"ucker embedding $\Gr(d,n) \subseteq \PP^{\binom{n}{d}-1}$ is maximal torus equivariant so induces a closed embedding of Chow quotients \[\Gr(d,n)\ChowQ S \subseteq \PP^{\binom{n}{d}-1}\ChowQ S\] for any diagonal subtorus $S \subseteq T$.   Since $S$ acts here through the dense torus for $\PP^{\binom{n}{d}-1}$, the Chow quotient $\PP^{\binom{n}{d}-1}\ChowQ S$ is a projective toric variety; the lattice polytope for it is a secondary polytope that we now describe (see \cite{KSZ91} and \cite[\S0.2]{Kap93Chow}).  

If we denote the coordinates on $\PP^{\binom{n}{d}-1}$ by $x_I, I\in \binom{[n]}{d}$ then $t = (t_1,\ldots,t_n) \in T$ acts by $t\cdot x_I = \left(\prod_{i\in I} t_i \right) x_I$.  These weights are encoded by the $\binom{n}{d}$ integer vectors $\sum_{i\in I} e_i \in \mathbb{Z}^n$.  The weights for the rank $m$ diagonal subtorus $S \subseteq T$ are then the images of these integer vectors under the linear map $\lambda_{\vec{r}} : \mathbb{R}^n \rightarrow \mathbb{R}^m$.  Define the following cardinality $\binom{n}{d}$ multiset: \[A := \{\lambda_{\vec{r}}\left(\sum_{i\in I} e_i\right)~|~I \in \binom{[n]}{d}\}.\]  The lattice polytope for $\PP^{\binom{n}{d}-1}\ChowQ S$ is the secondary polytope $\Sigma(A)$.  Recall that this means $\Sigma(A)$ is the convex hull in $\mathbb{R}^A$ of the characteristic functions $\varphi_{\mathcal{T}} : A \rightarrow \mathbb{Z}$ where $\mathcal{T}$ is a triangulation of the pair $(\Conv(A),A)$---meaning a collection of simplices, intersecting only along common faces, whose union is $\Conv(A)$ and whose vertices lie in $A$---and where by definition the value of $\varphi_{\mathcal{T}}$ on $a\in A$ is the sum of the volumes of all simplices in $\mathcal{T}$ for which $a$ is a vertex (with the volume form normalized by setting the volume of the smallest possible lattice simplex to be 1).

\subsection{$\Gr(2,4)$ with the maximal torus action}
Here $\vec{r} = (1,1,1,1)$ and $\lambda_{\vec{r}}$ is the identity on $\mathbb{R}^4$, so $A$ consists of the six vertices of the octahedron $\Delta(2,4)$, namely all permutations of the vector $(1,1,0,0)$.  There are three triangulations here: choose two of the three pairs of non-adjacent vertices and for each of these chosen pairs slice a plane through the remaining four vertices.  The three characteristic functions are then the vectors $(4,4,2,2,2,2)$, $(2,2,4,4,2,2)$, and $(2,2,2,2,4,4)$.  These form an equilateral triangle whose lattice points, in addition to the three vertices, are the midpoints of the three edges, namely $(3,3,3,3,2,2)$, $(3,3,2,2,3,3)$, and $(2,2,3,3,3,3)$.  This lattice polytope defines the toric variety $\mathbb{P}^2$ polarized by the line bundle $\mathcal{O}(2)$; by labeling the lattice points, in the order listed above, we can view this as $\Proj k[x^2,y^2,z^2,xy,xz,yz]$.  

The Grassmannian $\Gr(2,4)$ is a hypersurface in $\PP^5$, defined by a single Pl\"ucker relation, so the Chow quotient $\Gr(2,4)\ChowQ S \subseteq \PP^5\ChowQ S \cong \PP^2$ is also a hypersurface and our next task is finding the equation for it.  If we write the coordinates for $\PP^5$ as $(x_{12},x_{34},x_{13},x_{24},x_{14},x_{23})$ then the monomials specified by the six lattice points described in the preceding paragraph, after dividing by the common factor that is the product of the squares of all the variables, are the following:
\[m_1 = x_{12}^2x_{34}^2,~m_2 = x_{13}^2x_{24}^2,~m_3 = x_{14}^2x_{23}^2,~m_4 = x_{12}x_{34}x_{13}x_{24},~m_5 = x_{12}x_{34}x_{14}x_{23},~m_6 = x_{13}x_{24}x_{14}x_{23}.\]  Multiplying the Pl\"ucker relation
\[x_{12}x_{34} - x_{13}x_{24} + x_{14}x_{23} = 0\]
by $x_{12}x_{34}$ yields the relation $m_1 - m_4 + m_5$, and similarly multiplying by $x_{13}x_{24}$ yields $m_4 - m_2 + m_6 = 0$ and multiplying by $x_{14}x_{23}$ yields $m_5 - m_6 + m_3 = 0$.  These are linear relations among the monomials $m_i$, so they are three quadratic relations among the variables $x,y,z$ introduced at the end of the preceding paragraph, namely 
\[x^2 - xy + xz = 0,~ xy - y^2 + yz = 0,\text{ and } xz - yz + z^2 = 0.\]
These quadratics generate a non-saturated ideal whose saturation is the principal ideal generated by $x - y + z = 0$; this linear relation is the defining equation for the Chow quotient $\Gr(2,4)\ChowQ T \subseteq \PP^2$ that we were seeking.  Note that in \cite[Example 1.2.12]{Kap93Chow} Kapranov described this as a conic in the plane, whereas we see here more specifically it is a line in the plane together embedded by $\mathcal{O}(2)$ as a conic in the Veronese surface in $\PP^5$.

\subsection{$\Gr(2,4)$ with a rank 3 diagonal subtorus}

Now consider the rank 3 diagonal subtorus $S\subseteq T$ defined by $\vec{r} = (1,1,2)$, namely $S = \{(t_1,t_2,t_3,t_3)~|~t_i\in k^\times\}$, which acts on a subspace $L \in \Gr(2,4)$ represented by a $2\times 4$ matrix by rescaling the first two columns independently and rescaling the last two columns together.  Here $\Gr(2,4)\ChowQ S$ is a surface embedded in the toric threefold $\PP^5\ChowQ S$.  The linear map $\lambda_{\vec{r}} : \mathbb{R}^4 \rightarrow \mathbb{R}^3$ is $(x_1,x_2,x_3+x_4)$.  The multiset $A$, the image under this linear projection of the 6 vertices of the octahedron $\Delta(2,4)$, is $(1,1,0)$, $(1,0,1)$ with multiplicity 2, $(0,1,1)$ with multiplicity 2, and $(0,0,2)$.  This is a square with a pair of non-adjacent vertices doubled, and there are eight triangulations: there are two ways of subdividing with a diagonal line segment, and for each of these there are four ways of choosing which of the doubled vertices to use in the resulting pair of triangles.  The six characteristic functions, which we shall name $v_1,\ldots,v_4,w_1,\ldots,w_4$, then take the following form: 
\begin{eqnarray*}
v_1 = (1,2,0,2,0,1),~v_2 = (1,2,0,0,2,1),~v_3 = (1,0,2,0,2,1),~v_4 = (1,0,2,2,0,1),\\
w_1 = (2,1,0,1,0,2),~w_2 = (2,1,0,0,1,2),~w_3 = (2,0,1,0,1,2),~w_4 = (2,0,1,1,0,2).
\end{eqnarray*}
The convex hull of these is a 3-dimensional polytope.  The convex hull of the $v_i$ is a square and the convex hull of the $w_i$ is a smaller square that is parallel to it, so altogether we have a truncated square pyramid.  A square is the toric polytope description of $\PP^1\times\PP^1$, extending this to a square pyramid corresponds to taking the projective cone over $\PP^1\times\PP^1$, and truncating this pyramid corresponds to blowing up the torus-fixed cone point corresponding to the pyramid apex.  In coordinates this can be written \[\PP^5\ChowQ S \cong \Bl\left(\Proj k[x_0,x_1,x_2,x_3,y]/(x_0x_3-x_1x_2)\right),\] and by computing lattice lengths one sees that the polarization is $\mathcal{O}(2H-E)$.  To find the equations for the closed subvariety $\Gr(2,4)\ChowQ S$ inside here, we follow the approach in the previous example.  Plugging the variables $x_{ij}$ into the 8 vertices of our secondary polytope yields the following monomials:
\begin{eqnarray*}
m_1 = x_{12}x_{13}^2x_{23}^2x_{34},~m_2 = x_{12}x_{13}^2x_{24}^2x_{34},~m_3 = x_{12}x_{14}^2x_{24}^2x_{34},~m_4 = x_{12}x_{14}^2x_{23}^2x_{34},\\
n_1 = x_{12}^2x_{13}x_{23}x_{34}^2,~n_2 = x_{12}^2x_{13}x_{24}x_{34}^2,~n_3 = x_{12}^2x_{14}x_{24}x_{34}^2,~n_4 = x_{12}^2x_{14}x_{23}x_{34}^2.
\end{eqnarray*}
Multiplying the Pl\"ucker relation by $x_{12}x_{13}x_{24}x_{34}$ and by $x_{12}x_{14}x_{23}x_{34}$ yields the relations
\[n_2 - m_2 + \prod x_{ij} = 0 \text{ and }n_4 - \prod x_{ij} + m_4 = 0\] so our Grassmannian Chow quotient here is defined in the above toric Chow quotient by the single relation $m_2 - m_4 - n_2 - n_4 = 0$ in the polynomial Cox ring.

\subsection{$\Gr(2,4)$ with a balanced rank two diagonal subtorus}

Next, consider the diagonal subtorus $\{(t_1,t_1,t_2,t_2)~|~t_i\in k^\times\}$ defined by $\vec{r}=(2,2)$.  The linear projection $\lambda_{\vec{r}} : \mathbb{R}^4 \rightarrow \mathbb{R}^2$ is $(x_1+x_2,x_3+x_4)$ which sends the vertices of $\Delta(2,4)$ to $(2,0)$, $(1,1)$ four times, and $(0,2)$.   The result of course is an interval with a single interior lattice point that has been quadrupled.  There are five triangulation, four from subdividing with the different midpoints and one from not subdividing at all; the characteristic functions are:
\[v_1 = (1,2,0,0,0,1),v_2=(1,0,2,0,0,1),v_3=(1,0,0,2,0,1),v_4=(1,0,0,0,2,1),v_5=(2,0,0,0,0,2).\]
The convex hull of $v_1,\ldots,v_4$ is a tetrahedron giving the polarized toric variety $(\PP^3,\mathcal{O}(2))$, and $\\P^5\ChowQ S$ is the toric variety given by the convex cone over this tetrahedron with apex $v_5$.  Plugging the variables $x_{ij}$ into these five vertices yields
\[m_1 = x_{12}x_{13}^2x_{34}, ~m_2 = x_{12}x_{14}^2x_{34}, ~m_3 = x_{12}x_{23}^2x_{34}, ~m_4 = x_{12}x_{24}^2x_{34}, ~m_5 = x_{12}^2x_{34}^2.\]
The Pl\"ucker relation can be expressed as
\[\sqrt{m_5} - \sqrt{\frac{m_1m_4}{m_5}} + \sqrt{\frac{m_2m_3}{m_5}} = 0,\]
which after some elementary algebra yields the relation
\[m_1^2m_4^2 + m_2^2m_3^2 + m_5^2 - 2m_1m_2m_3m_4 - 2m_1m_4m_5^2 - 2m_2m_3m_5^2 = 0\]
defining $\Gr(2,4)\ChowQ S$ in the Cox ring of our toric variety $\PP^5\ChowQ S$.

\section{Maps between Chow quotients}

Let us start here by generalizing Kapranov's \cite[Theorem 1.6.6]{Kap93Chow}; while one probably could have adapted Kapranov's proof nearly verbatim to our setting, we instead provide a slight variant that we feel brings out more prominently the elegant toric geometry underlying the result.

\begin{proof}[Proof of Theorem \ref{thm:hypersimplicial}]
Recall from the theorem statement that we have fixed an index $i$ and denoted by $I$ the index of the $r_i$ columns acted upon nontrivially by the $i^{\text{th}}$ $\Gm$ factor of $S$ and by $S_i$ the projection of $S$ onto the coordinates outside of $I$.  So $S$ has rank $m$ and $S_i$ has rank $m-1$.  Let 
\[a_i : \Gr(d,n)\ChowQ S \dashrightarrow \Gr(d-r_i,n-r_i)\ChowQ S_i\] 
be the rational map sending a generic torus orbit closure $\overline{S\cdot L}$, $L\in \Gr(d,n)$, to the torus orbit closure $\overline{S_i\cdot (L\cap H_I)}$, where $H_I \subseteq k^n$ is the coordinate linear subspace defined by setting all coordinates in $I$ equal to zero (and $\Gr(d-r_i,n-r_i)$ here parameterizes subspaces of $H_I \cong k^{n-{r-i}}$).  Let
\[b_i : \Gr(d,n)\ChowQ S \dashrightarrow \Gr(d,n-r_i)\ChowQ S_i\] 
be the rational map sending a generic $\overline{S\cdot L}$ to $\overline{S_i\cdot \pi_{I^c}(L)}$, where $\pi_{I^c} : k^n \rightarrow k^{n-r_i}$ projects away the $I$-coordinates.

To show that these rational maps extend to morphisms, we will use the valuative criterion provided in \cite[Theorem 7.3]{GG14Kap} (here for convenience we will use the analytic language of 1-parameter families, rather valuation rings, since we have restricted to the setting $k=\mathbb{C}$ anyway).  This means we need to show that for any 1-parameter family of cycles $Z_t$, $t\in k^\times$, in the interior of $\Gr(d,n)\ChowQ S$, which necessarily maps to a 1-parameter family of cycles $a_i(Z_t)$ in the interior of $\Gr(d-r_i,n-r_i)\ChowQ S_i$, the limit cycle \[\lim_{t \rightarrow 0}a_i(Z_t) \in \Gr(d-r_i,n-r_i)\ChowQ S_i \subseteq \Chow\left(\Gr(d-r_i,n-r_i)\right) \subseteq \Chow\left(\PP^{\binom{n-r_i}{d-r_i}-1}\right)\] depends only on the limit cycle
\[Z_ 0 := \lim_{t \rightarrow 0}Z_t \in \Gr(d,n)\ChowQ S \subseteq \Chow\left(\Gr(d,n)\right) \subseteq \Chow\left(\PP^{\binom{n}{d}-1}\right),\] and similarly for $b_i$.  We will do this by explicitly describing $\lim a_i(Z_0)$ and $\lim b_i(Z_t)$ in terms of $Z_0$.

Following Kapranov, let $G_j^+ \subseteq \Gr(d,n)$ be the locus of linear subspaces containing the $j^{\text{th}}$ coordinate axis, and let $G_j^- \subseteq \Gr(d,n)$ be the locus of linear subspaces contained in the hyperplane where the $j^{\text{th}}$ coordinate is zero.  Then, as noted in \cite[Proposition 1.6.10]{Kap93Chow}, \[\Gr(d-1,n-1) \cong G_j^+ = \Gr(d,n) \cap \Pi_j^+\] where $\Pi_j^+ \subseteq \PP^{\binom{n}{d}-1}$ is the coordinate linear subspace defined by $x_J = 0$ for $J \not\owns j$, and \[\Gr(d,n-1) \cong G_j^- =  \Gr(d,n)\cap \Pi_j^-\] where $\Pi_j^-$ is the coordinate linear subspace defined by $x_J = 0$ for $J \owns j$.  In our setting we shall need to consider certain intersections of these sub-Grassmannians, so let
\[\Pi_I^{\pm} := \bigcap_{j\in I}\Pi_j^{\pm}\text{ and }G_I^{\pm} := \bigcap_{j\in I}G_j^{\pm} = Gr(d,n) \cap \Pi_I^{\pm}.\]

We claim that 
\[\lim_{t\rightarrow 0}a_i(Z_t) = Z_0 \cap \Pi_I^+\text{ and }\lim_{t\rightarrow 0}b_i(Z_t) = Z_0 \cap \Pi_I^-.\]
Verifying this claim will establish the theorem, by the aforementioned valuative criterion.

The argument in Kapranov's \cite[Lemma 1.6.13]{Kap93Chow} applies equally well for diagonal subtori and shows that for $t\ne 0$ we have $a_i(Z_t) = Z_t \cap \Pi_I^+$ and $b_i(Z_t) = Z_t \cap \Pi_I^-$, and from this it immediately follows from elementary topology that 
\begin{equation}\label{eq:limits}
\lim_{t\rightarrow 0}a_i(Z_t) \subseteq Z_0 \cap \Pi_I^+\text{ and }\lim_{t\rightarrow 0}b_i(Z_t) \subseteq Z_0 \cap \Pi_I^-,
\end{equation}
We claim that in both cases the intersection on the right has the same dimension as the limit on the left, namely $m-2$ (the diagonal $\Gm$ where all torus coordinates are equal acts trivially so a full-dimensional orbit has dimension one less than the rank of the torus).  To see, first note that by Proposition \ref{prop:cycles} we can write $Z_0 = \sum_{j=1}^\ell \overline{S\cdot L_j}$ for linear subspaces $L_j$ whose $S$-orbits have full dimension $m-1$.  Then \[Z_0 \cap \Pi_I^{\pm} = \sum_{j=1}^\ell\left(\overline{S\cdot L_j} \cap \Pi_I^{\pm}\right).\]  If the dimension of this intersection were not equal to $m-2$ it would have to be dimension $m-1$, the dimension of $Z_0$, which means for at least one $j$ we would have $L_j \subseteq \Pi_I^{\pm}$,   But this would mean that the $S$-orbit of this $L_j$ is not full-dimensional, contradicting our assumption on it.  Indeed, if $L_j \subseteq \Pi_I^+$ then the rank one subtorus of $S$ where all $\Gm$ factors except for the $i^{\text{th}}$ are trivial is in the stabilizer of $L_j$, since this $\Gm$ subtorus rescales equally by $t^{r_i}$ the Pl\"ucker coordinates $x_J$ where $J \supseteq I$ and by definition of $\Pi_I^+$ all remaining Pl\"ucker coordinates are zero; similarly, if $L_j \subseteq \Pi_I^-$ then this same $\Gm$ factor is in the stabilizer of $L_j$, since here it acts trivially on the Pl\"ucker coordinates $x_J$ where $J \cap I = \varnothing$ and by definition of $\Pi_I^-$ all remaining Pl\"ucker coordinates are zero.

For each of the containments in Equation \eqref{eq:limits}, since the dimensions of both sides are equal, to prove that the containment is an equality it suffices to prove that the degrees of both sides are equal.  Now, $\lim_{t\rightarrow 0} a_i(Z_t)$ is a limit of generic $S_i$-orbit closures so it has the same degree as a generic orbit closure $\overline{S_i\cdot L}$, $L\in\Gr(d-r_i,n-r_i)^0$.  But $\overline{S_i\cdot L}$ is a toric variety so its degree is the volume of the lattice polytope $\Delta_{\overline{S_i\cdot L}}$, and since $L$ here is generic this lattice polytope is the full linearly projected hypersimplex $\lambda_{\pi_{[m]\setminus i}\vec{r}}\left(\Delta(d-r_i,n-r_i)\right)$, where $\lambda_{\pi_{[m]\setminus i}\vec{r}} : \mathbb{R}^{n-r_i} \rightarrow \mathbb{R}^{m-1}$ is the linear projection map corresponding to the diagonal subtorus $S_i$ of the maximal torus acting on $\Gr(d-r_i,n-r_i)$.  On the other hand, by Proposition \ref{prop:cycles} for the limit cycle $Z_0 = \sum_{j=1}^\ell \overline{S\cdot L_j}$ the lattice polytopes $\Delta_{\overline{S\cdot L_1}},\ldots,\Delta_{\overline{S\cdot L_\ell}}$ form a polyhedral decomposition of $\lambda_{\vec{r}}\left(\Delta(d,n)\right)$.  Then the lattice polytopes $\Delta_{\overline{S\cdot L_1}} \cap \lambda_{\vec{r}}(\Gamma_I^+),\ldots,\Delta_{\overline{S\cdot L_\ell}}\cap \lambda_{\vec{r}}(\Gamma_I^+)$ form a polyhedral decomposition of the face $\lambda_{\vec{r}}(\Gamma_I^+)$ of $\lambda_{\vec{r}}(\Delta(d,n))$, where $\Gamma_I^+ := \cap_{j\in I}\Gamma_j^+$ and $\Gamma_j^+$ is the face of $\Delta(d,n)$ that Kapranov identified in \cite[Proposition 1.6.10]{Kap93Chow} as the image under the moment map $\mu_T$ of $G_j^+ \subseteq \Gr(d,n)$.  We claim
\begin{eqnarray*}
\deg(Z_0 \cap \Pi_I^+) \le \sum_{j=1}^\ell \deg\left(\overline{S\cdot L_j} \cap \Pi_I^+ \right) = \sum_{j=1}^\ell \vol\left(\Delta_{\overline{S\cdot L_j}}\cap \Gamma_I^+\right) = \vol\left(\lambda_{\vec{r}}(\Gamma_I^+)\right)\\
 = \vol\left(\lambda_{\pi_{[m]\setminus i}\vec{r}}\left(\Delta(d-r_i,n-r_i)\right)\right).
\end{eqnarray*}
Indeed, the inequality here allows for the possibility that some of these intersected orbit closures are not full-dimensional, the first equality is Kapranov's observation in \cite[Proposition 1.6.10]{Kap93Chow} about the interplay between the moment map and the sub-Grassmannians $G_j^+$, the second equality is due to the above observation about having a polyhedral decomposition, and the final equality follows from the observation that the moment map $\mu_S$ restricted to the sub-Grassmannian $\Gamma_I^+ \cong \Gr(d-r_i,n-r_i)$ is identified by this isomorphism with the moment map $\mu_{S_i} = \lambda_{\pi_{[m]\setminus i}\vec{r}}\circ \mu_T'$ where $T'$ is the maximal torus acting on $\Gr(d-r_i,n-r_i)$.  This concludes the argument for $a_i$, and the volume calculation for $b_i$ is entirely analogous.
\end{proof}

\section{Generalized Gelfand-MacPherson correspondence and Gale duality}

In \cite{Tha99} Thaddeus studies an interesting classical geometric situation related to the configuration spaces studied by Kapranov in \cite{Kap93Chow}, and while doing so he proves a handful of results that are in close analogy with results in Kapranov's paper---but in almost all cases, the proofs Thaddeus provides are new, not merely adaptations of Kapranov's.  In particular, when studying Chow quotients Thaddeus avails himself of the functorial machinery developed by Koll\'{a}r in \cite{Kollar-chow}, obviating the need to rely on the analytic methods for working with Chow varieties that were the only option for Kapranov at the time his paper was written.  We adapt here one particular proof of Thaddeus (and a particularly clever one at that) which in our setting yields the generalized Gelfand-MacPherson isomorphism Theorem \ref{thm:genGM} stated in the introduction.  Note that by specializing to the maximal torus this yields an explicit Thaddeus-esque proof of Kapranov's original Chow-theoretic Gelfand-MacPherson isomorphism \cite[Theorem 2.2.4]{Kap93Chow}.

\begin{proof}[Proof of Theorem \ref{thm:genGM}]
The basic idea is, quite like the usual Gelfand-MacPherson correspondence, to observe that the $\GL_d$-action on the affine space of $n\times d$ matrices (we have taken a transpose here to work with sub rather than quotient objects, but that is immaterial and just to ease notation) commutes with the torus action; taking the $\GL_d$ quotient first yields the Grassmannian $\Gr(d,n)$, whereas taking the $S$-quotient first projectivizes the size $r_i \times d$ matrix blocks, $i=1,\ldots,m$, of this space of matrices resulting in a product of projective spaces.  In fact, this already shows that the two sides of the claimed isomorphism are birational, so the work is to extend this birational map to an isomorphism.  To do this, we follow and mildly adapt the argument of Thaddeus in his proof in \cite[\S6.3]{Tha99}.  The main insight in Thaddeus' proof, translated to our situation, is that the two rational quotient maps 
\begin{equation}\label{eq:SLquot}\PP\Hom(k^d,k^n) \dashrightarrow \Gr(d,n)\end{equation} 
and 
\begin{equation}\label{eq:Tquot}\PP\Hom(k^{d},k^{n}) \dashrightarrow \prod_{i=1}^m\PP\Hom(k^d,k^{r_i}) \end{equation}
have different base loci, and by resolving both it is easier to compare cycles by using pullback and pushforward properties of the Chow variety.  We now go through these details in earnest.

The rational $\GL_d$-quotient map in \eqref{eq:SLquot} sends an injective linear map $\varphi : k^d \hookrightarrow k^n$ to $[\varphi(k^d)]\in \Gr(d,n)$, the point in the Grassmannian corresponding to the image of this linear map; the base locus is the set of linear maps $k^d \rightarrow k^n$ with nontrivial kernel.  Let $\mathcal{S}_{d,n} \rightarrow \Gr(d,n)$ denote the universal sub-bundle over the Grassmannian.  Then the rational $\GL_d$-quotient map is resolved by the space $\PP\Hom(k^{d},\mathcal{S}_{d,n})$:
\[\xymatrix{ & \PP\Hom(k^d,\mathcal{S}_{d,n})\ar[dr]\ar[dl] & \\ \PP\Hom(k^{d},k^{n}) \ar@{-->}[rr] & & \Gr(d,n)}\]
Indeed, the fiber over a point $\varphi : k^d \rightarrow k^n$ of $\PP\Hom(k^d,k^n)$ is a single point of $\PP\Hom(k^d,\mathcal{S}_{d,n})$ if $\varphi$ is injective, namely $\varphi$ viewed as a map from $k^d$ to its image $\varphi(k^d) \subseteq k^n$, whereas if $\dim\varphi(k^d) < d$ then the fiber in $\PP\Hom(k^d,\mathcal{S}_{d,n})$ is in bijection with all $d$-dimensional subspaces $L\subseteq k^n$ containing $\varphi(k^d) \subseteq k^n$, since for each such $L \supseteq \varphi(k^d)$ we have the element of the fiber given by viewing $\varphi$ as a map from $k^d$ to $L$.  In fact, $\PP\Hom(k^d,\mathcal{S}_{d,n})$ is the iterated blow-up of $\PP\Hom(k^d,k^n)$ along the locus of non-full rank maps, ordered in increasing order of rank.  Note that the morphism to $Gr(d,n)$ is a $\PP^{d^2-1}$-bundle; in particular, it is flat.

On the other hand, the rational $S$-quotient map (\ref{eq:Tquot}) is resolved by the $\PP^{m-1}$-bundle given by the projectivization of the total space the direct sum of the dual line bundles to the tautological bundles:
\[\xymatrix{ & \PP\left(\bigoplus_{i=1}^m\mathcal{O}(e_i)\right)\ar[dr]\ar[dl] & \\ \PP\Hom(k^{d},k^{n}) \ar@{-->}[rr] & & \prod_{i=1}^m\PP\Hom(k^d,k^{r_i})}\]
Here $\mathcal{O}(e_j)$ denotes the pull-back of $\mathcal{O}(1)$ along the $j^{\text{th}}$ projection \[\prod_{i=1}^m\PP\Hom(k^d,k^{r_i}) \rightarrow \PP\Hom(k^d,k^{r_j}) \cong \PP^{r_jd-1}.\]  
One can see this as follows.  The base locus for this map consists of matrices where any of the $r_i\times d$ blocks (corresponding to the diagonal subtorus action) are entirely zero, so to resolve this map we need to blow up this locus.  Since it is a union of linear subspaces meeting transversely, this can be done one subspace at a time, in any order, and we thus reduce to the standard observation that the total space of $\mathcal{O}(1)$ on any projective space $\PP^{\ell}$ is the blow-up of $\mathbb{A}^{\ell+1}$ at the origin.  

Putting this together, we get the following commutative diagram:
\[\xymatrix{
\PP\left(\bigoplus_{i=1}^m\mathcal{O}(e_i)\right)\ar[dd]\ar[dr] & & \PP\Hom(k^{d},\mathcal{S}_{d,n})\ar[dl]\ar[dd]\\
 & \PP\Hom(k^{d},k^{n})\ar@{-->}[dl]\ar@{-->}[dr] & \\
\prod_{i=1}^m\PP\Hom(k^d,k^{r_i})\ar@{-->}[d] & & Gr(d,n)\ar@{-->}[d] \\
\left(\prod_{i=1}^m\PP\Hom(k^d,k^{r_i})\right)\ChowQ\GL_{d} & & Gr(d,n)\ChowQ S
}\]
Here the vertical morphisms are both projective space bundles, the diagonal morphisms are birational, and the dashed arrows are all rational quotient maps---on the left by the torus first then $\GL_{d}$, and on the right by $\GL_{d}$ first then the torus.

The rest of Thaddeus' argument now goes through essentially verbatim.  The universal family of cycles on $\Gr(d,n)$ over the Chow quotient $\Gr(d,n)\ChowQ S \subseteq \Chow\left(\Gr(d,n)\right)$ pulls back along the flat morphism to a family of cycles on $\PP\Hom(k^d,\mathcal{S}_{d,n})$ over $\Gr(d,n)\ChowQ S$ with general fiber a $(\GL_{d}\times S)$-orbit closure.  This family pushes forward along the birational morphism to an $S$-invariant family of cycles on $\PP\Hom(k^d,k^n)$.  The restriction of the cycles in this family to the complement of the base locus of the torus quotient map (\ref{eq:Tquot}) pushes forward along this quotient map, a geometric quotient, and yields a family of cycles on $\prod_{i=1}^m\PP\Hom(k^d,k^{r_i})$ over $\Gr(d,n)\ChowQ S$.  Since $\Gr(d,n)\ChowQ S$ is reduced and the cycles over it in this last family all have the expected dimension, there is an induced morphism \[\Gr(d,n)\ChowQ S \rightarrow \Chow\left(\prod_{i=1}^m\PP\Hom(k^d,k^{r_i})\right)\]
by \cite[Theorem 3.21]{Kollar-chow}.  A general point of this Chow quotient gets sent to a $\GL_{d}$-orbit closure, so the image of this morphism is contained in the Chow quotient $\left(\prod_{i=1}^m\PP\Hom(k^d,k^{r_i})\right)\ChowQ\GL_{d}$.  On the other hand, the same argument applied symmetrically to other side of the above big commutative diagram yields a morphism between these Chow quotients in the other direction.  Since these Chow quotients are separated varieties, to show that these morphisms are inverse to each other, and hence that the two Chow quotients are isomorphic, it suffices to show that they are inverse on open dense loci.  For this we apply the naive argument discussed at the beginning of this proof, regarding commuting group actions, to see that indeed these maps identify generic orbit closures.  
\end{proof}

An immediate corollary of this is the generalized Gale duality Corollary \ref{cor:genGale} stated in the introduction.  Indeed, the orthogonal complement isomorphism $\Gr(d,n) \cong \Gr(n-d,n)$ is torus-equivariant so descends to an isomorphism $\Gr(d,n)\ChowQ S \cong \Gr(n-d,n)\ChowQ S$ of Chow quotients for any subtorus $S$, and applying our generalized Gelfand-MacPherson isomorphisms to both sides of this isomorphism provides our generalized Gale duality isomorphism.

\begin{remark}
For parameters $m,d,d_1,\ldots,d_m$ such that
\[2d - m = \sum_{i=1}^m d_i,\]
our generalized Gale duality sends configurations of $m$ parameterized linear subspaces of dimensions $d_1,\ldots,d_m$ in $\PP^d$ to configurations of $m$ parameterized linear subspaces of dimensions $d_1,\ldots,d_m$ in $\PP^d$, so in this situation one could study ``self-associated'' configurations, generalizing the maximal torus case studied by Kapranov in \cite[Paragraph (2.3.9)]{Kap93Chow} (see also \cite[\S II]{Eisenbud-Popescu} for another setting for self-association).
\end{remark}

\section{The Borel transfer principle and the Chen-Gibney-Krashen moduli space }

Consider a connected unipotent group $H$, and suppose $G$ is a reductive group containing $H$ as a closed subgroup.  The quotient $G/H$, where $H$ acts on the right, is a quasi-affine variety (and if $H$ is positive-dimensional then it is not affine); it admits a natural embedding in the affinization \[(G/H)^\text{aff} := \Spec \mathcal{O}_{G/H}(G/H) = \Spec \mathcal{O}_G(G)^H\] which is a scheme possibly of infinite type since the ring of invariants of a non-reductive group need not be finitely generated.

\begin{example}\label{ex:nonred}
Let \[G := \Spec k[x_{11},x_{12},x_{21},x_{22}, (x_{11}x_{22} - x_{12}x_{21})^{-1}] \cong \GL_2\] be the affine group variety of $2\times 2$ invertible matrices, and let $H := \Spec k[s] \cong \Ga$ be the subgroup of unipotent matrices of the form $\left(\begin{array}{cc}1 & s \\0 & 1\end{array}\right)$.  The quotient $G/H$ is the quasi-affine variety $\mathbb{A}^2\setminus\{0\}$, because the affinization is \[(G/H)^\text{aff} = \Spec k[x_{11},x_{12},x_{21},x_{22}, (x_{11}x_{22} - x_{12}x_{21})^{-1}]^H = \Spec k[x_{11},x_{21}] \cong \mathbb{A}^2\] but the image of the quotient morphism $G \rightarrow (G/H)^\text{aff}$ does not include the origin since a matrix where $x_{11}$ and $x_{21}$ are both zero is not invertible.  In this case the affinization is of finite type.
\end{example}

Continue to let $G$ and $H$ be a reductive group and unipotent subgroup as above, and suppose now that $X$ is an affine variety with an $H$-action that extends to a $G$-action.  The classical Borel transfer principle states, in the language of (non-reductive) GIT, that there is an isomorphism \[X\quotient H \cong \left((G/H)^\text{aff} \times X\right)\quotient G,\] where $G$ acts diagonally on this product, with the $G$-action on $(G/H)^\text{aff}$ induced by left-multiplication of $G$ on itself, and the symbol ``$\quotient$'' simply means to take Spec of the ring of invariants \cite[\S5.1]{Doran-Kirwan}.  This allows one to replace a non-reductive invariant ring with a reductive invariant ring, though in the process one replaces the $k$-algebra being acted upon with one that need not be finitely generated.  This is often a useful tradeoff as it means instead of studying the $H$-action of $X$, it suffices to study the typically simpler $H$-action on $G$ together with the (again, typically simpler) $G$-action on $X$.  This geometric formulation of the Borel transfer principle has been globalized to the case that $X$ is projective in \cite[\S5.1]{Doran-Kirwan}.  

The definition of a Chow quotient is perfectly valid for any algebraic group, not just reductive groups, so a natural question, which seems not to have appeared in the literature previously, is whether this global Borel transfer principle for GIT quotients extends to Chow quotients:
\begin{question}\label{q:borel}
Let $G$ be a reductive group containing a connected unipotent closed subgroup $H$, let $\overline{G/H}$ be a projective completion of the quotient, and let $X$ be a projective variety with a $G$-action.  Is there an isomorphism 
\[X\ChowQ H \stackrel{?}{\cong} \left(\overline{G/H} \times X\right)\ChowQ G,\]
or at least are there reasonable hypotheses guaranteeing such an isomorphism?
\end{question}
The projective completion here is needed since Chow varieties are only defined for projective varieties.  In what follows we show that a specific instance of this question is an open question about a Grassmannian Chow quotient first asked (in casual conversation) by Krashen.

Consider the diagonal subtorus action on $\Gr(d,n)$ defined by $\vec{r} = (d-1,1,\ldots,1)$, so that $S$ is the rank $n-d+2$ torus that acts by rescaling the first $d-1$ columns of a matrix together and the last $n-d+1$ columns individually.  By our generalized Gelfand-MacPherson isomorphism (Theorem \ref{thm:genGM}) we have
\begin{equation}\label{eq:ChowQ}
\Gr(d,n)\ChowQ S \cong \left(\PP\Hom(k^{d-1},k^d) \times (\PP^{d-1})^{n-d+1}\right)\ChowQ \GL_d,
\end{equation}
a compactification of the configuration space of $n-d+1$ points and a parameterized hyperplane in $\PP^{d-1}$.  On the other hand, the Chen-Gibney-Krashen moduli space $T_{d-1,n-d+1}$ is a compactification of the same configuration space \cite{CGK05}, and Krashen's question is whether these are isomorphic.  In \cite{Gallardo-Giansiracusa} it is shown that $T_{d-1,n-d+1}$ is isomorphic to the normalization of the Chow quotient $(\PP^{d-1})^{n-d+1}\ChowQ H$, where $H \cong \Gm^2 \rtimes \Ga^{d-1}$ is the non-reductive subgroup of $\GL_d$ fixing a hyperplane pointwise.  Since this $H$-action extends to the standard $\GL_d$-action, we can apply Question \ref{q:borel} and ask whether this non-reductive Chow quotient is isomorphic to the reductive Chow quotient $\left(\overline{\GL_d/H} \times (\PP^{d-1})^{n-d+1}\right)\quotient \GL_d$.  The following lemma describes $\overline{\GL_d/H}$ and the induced group actions and implies that this reductive Chow quotient is precisely the one appearing in our generalized Gelfand-MacPherson correspondence, the right side of Equation \eqref{eq:ChowQ}, and hence as claimed that the Krashen question is a specific instance of Question \ref{q:borel}:

\[\xymatrix{ T_{d-1,n-d+1} \ar@{=}[d] \ar@{=}[r]^? & \Gr(d,n)\ChowQ S \ar@{=}[d] \\ (\PP^{d-1})^{n-d+1} \ChowQ H \ar@{=}[r]_(0.39){?} & \left(\overline{\GL_d/H} \times (\PP^{d-1})^{n-d+1}\right)\quotient \GL_d }\]

The left vertical equality (up to normalization) here is \cite{Gallardo-Giansiracusa}, the right vertical equality is the following lemma together with the Gelfand-MacPherson isomorphism, the top horizontal equality is the Krashen question, and the bottom horizontal equality is a special instance of Question \ref{q:borel}.

\begin{lemma}\label{lem:Borel}
For the right-multiplication action of $H$ on $\GL_{d}$, the quotient $\GL_d/H$ is isomorphic to the open subvariety of $\PP\Hom(k^{d-1},k^{d})$ consisting of projective equivalence classes of full rank $d\times (d-1)$ matrices.  The left-multiplication action of $\GL_d$ on itself descends to an action on this quotient corresponding, via this isomorphism, to left matrix multiplication.
\end{lemma}

Certainly the most natural projective completion to take for the space of full rank matrices is its Zariski closure in the space of all matrices, hence $\overline{\GL_d/H} = \PP\Hom(k^{d-1},k^d)$.

\begin{proof}
If we choose coordinates so that the fixed hyperplane is defined by the vanishing of the first coordinate, then $H \cong \Gm^2 \rtimes \Ga^{d-1}$ consists of matrices of the form
\[\left(\begin{array}{cccc}t_1 &  0 & \cdots & 0 \\s_1 & t_2 &  & 0 \\
\vdots &  & \ddots &  \\s_{d-1} & 0  & \cdots & t_2 \end{array} 
\right)\] 
for $s_i \in k$ and $t_i\in k^\times$.

Since the additive action is normalized by the torus action, we can compute the quotient in stages:
\[\GL_d/H \cong (\GL_d/\Ga^{d-1})/\Gm^2.\]
We claim $\GL_d/\Ga^{d-1}$ is the space of full rank $d\times (d-1)$ matrices.  Indeed, by viewing \[\GL_d \subseteq \Hom(k^d,k^d) \cong \mathbb{A}^{d^2}\] as the affine open complement of the hypersurface $\det = 0$, the ring of invariants for the $\Ga^{d-1}$-action is generated by all entries of the matrix except for those of the first column.  Thus the categorical quotient, in the category of affine varieties, is \[\GL_d\quotient \Ga^{d-1} \cong \Hom(k^{d-1},k^{d}) \cong \mathbb{A}^{(d-1)d}.\] However, similar to the situation in Example \ref{ex:nonred}, since this is a non-reductive quotient the quotient morphism need not be surjective, and indeed in the present situation its image is manifestly the set of full rank matrices.  

The residual $\Gm^2$-action on this space of full rank $d\times (d-1)$ matrices has the $\Gm$ factor corresponding to $t_1$ acting trivially and the $\Gm$ factor corresponding to $t_2$ acting by rescaling all entries equally, so the quotient by $\Gm^2$ is simply the projectivization.  The assertion about the induced left-multiplication action of $GL_d$ on this space of matrices follows immediately from our explicit description of the quotient in terms of invariants as the rightmost $d-1$ columns of a square $d\times d$ matrix of indeterminates.
\end{proof}

\bibliographystyle{amsalpha}
\bibliography{Diag}

\begin{bibdiv}
\begin{biblist}

\bib{Volume}{article}{
      author={Ardila, Federico},
      author={Benedetti, Carolina},
      author={Doker, Jeffrey},
       title={Matroid polytopes and their volumes},
        date={2010},
     journal={Discrete Comput. Geom.},
      volume={43},
      number={4},
       pages={841\ndash 854},
}

\bib{Alexeev-weighted}{book}{
      author={Alexeev, Valery},
       title={Moduli of weighted hyperplane arrangements},
      series={Advanced Courses in Mathematics. CRM Barcelona},
   publisher={Birkh\"auser/Springer, Basel},
        date={2015},
        note={Edited by Gilberto Bini, Mart\'\i Lahoz, Emanuele Macr\`\i and
  Paolo Stellari},
}

\bib{CGK05}{article}{
      author={Chen, L.},
      author={Gibney, A.},
      author={Krashen, D.},
       title={Pointed trees of projective spaces},
        date={2009},
     journal={J. Algebraic Geom.},
      volume={18},
      number={3},
       pages={477\ndash 509},
}

\bib{Doran-Kirwan}{article}{
      author={Doran, Brent},
      author={Kirwan, Frances},
       title={Towards non-reductive geometric invariant theory},
        date={2007},
        ISSN={1558-8599},
     journal={Pure Appl. Math. Q.},
      volume={3},
      number={1, Special Issue: In honor of Robert D. MacPherson. Part 3},
       pages={61\ndash 105},
}

\bib{Eisenbud-Popescu}{article}{
      author={Eisenbud, David},
      author={Popescu, Sorin},
       title={The projective geometry of the {G}ale transform},
        date={2000},
     journal={J. Algebra},
      volume={230},
      number={1},
       pages={127\ndash 173},
}

\bib{Fulton-MacPherson}{article}{
      author={Fulton, William},
      author={MacPherson, Robert},
       title={A compactification of configuration spaces},
        date={1994},
        ISSN={0003-486X},
     journal={Ann. of Math. (2)},
      volume={139},
      number={1},
       pages={183\ndash 225},
}

\bib{GG14Kap}{article}{
      author={Giansiracusa, Noah},
      author={Gillam, William~Danny},
       title={On {K}apranov's description of {$\overline M_{0,n}$} as a {C}how
  quotient},
        date={2014},
     journal={Turkish J. Math.},
      volume={38},
      number={4},
       pages={625\ndash 648},
}

\bib{Gallardo-Giansiracusa}{article}{
      author={Gallardo, Patricio},
      author={Giansiracusa, Noah},
       title={Modular interpretation of a non-reductive {C}how quotient},
        date={2018},
        ISSN={0013-0915},
     journal={Proc. Edinb. Math. Soc. (2)},
      volume={61},
      number={2},
       pages={457\ndash 477},
}

\bib{GGMS87}{article}{
      author={Gelfand, I.~M.},
      author={Goresky, R.~M.},
      author={MacPherson, R.~D.},
      author={Serganova, V.~V.},
       title={Combinatorial geometries, convex polyhedra, and {S}chubert
  cells},
        date={1987},
     journal={Adv. in Math.},
      volume={63},
      number={3},
       pages={301\ndash 316},
}

\bib{HH02}{article}{
      author={Herzog, J\"urgen},
      author={Hibi, Takayuki},
       title={Discrete polymatroids},
        date={2002},
     journal={J. Algebraic Combin.},
      volume={16},
      number={3},
       pages={239\ndash 268},
}

\bib{HKT}{article}{
      author={Hacking, Paul},
      author={Keel, Sean},
      author={Tevelev, Jenia},
       title={Compactification of the moduli space of hyperplane arrangements},
        date={2006},
     journal={J. Algebraic Geom.},
      volume={15},
      number={4},
       pages={657\ndash 680},
}

\bib{Kap93Chow}{book}{
      author={Kapranov, M.~M.},
       title={Chow quotients of {G}rassmannians. {I}},
      series={Adv. Soviet Math.},
   publisher={Amer. Math. Soc., Providence, RI},
        date={1993},
      volume={16},
}

\bib{Kollar-chow}{book}{
      author={Koll\'{a}r, J\'{a}nos},
       title={Rational curves on algebraic varieties},
      series={Ergebnisse der Mathematik und ihrer Grenzgebiete. 3. Folge. A
  Series of Modern Surveys in Mathematics [Results in Mathematics and Related
  Areas. 3rd Series. A Series of Modern Surveys in Mathematics]},
   publisher={Springer-Verlag, Berlin},
        date={1996},
      volume={32},
        ISBN={3-540-60168-6},
}

\bib{KSZ91}{article}{
      author={Kapranov, M.~M.},
      author={Sturmfels, B.},
      author={Zelevinsky, A.~V.},
       title={Quotients of toric varieties},
        date={1991},
     journal={Math. Ann.},
      volume={290},
      number={4},
       pages={643\ndash 655},
}

\bib{Tha99}{article}{
      author={Thaddeus, Michael},
       title={Complete collineations revisited},
        date={1999},
     journal={Math. Ann.},
      volume={315},
      number={3},
       pages={469\ndash 495},
}

\end{biblist}
\end{bibdiv}

\end{document}